\documentclass[11pt]{article}
\usepackage{bm}
\usepackage{tikz}
\usepackage[hidelinks]{hyperref}
\usepackage{indentfirst}
\usepackage{caption}
\usepackage{subcaption}
\usepackage{amsthm,amsmath,amssymb,mathrsfs,graphicx,color,url}
\theoremstyle{plain}
\newtheorem{thm}{Theorem}[section]

\theoremstyle{definition}

\newtheorem{rem}[thm]{Remark}

\numberwithin{equation}{section}
\def\E{\mathbb{E}}
\def\N{\mathbb{N}}
\def\P{\mathbb{P}}

\def\R{\mathbb{R}}

\def\inlaw{\stackrel{d}{=}}

\title{Functional limit theorems for elephant random walks on general periodic structures} 
\author{Shuhei Shibata\thanks{Joint Graduate School of Mathematics for Innovation, Kyushu
   University, Fukuoka, 819-0395, JAPAN.
   \textit{E-mail address}: \texttt{shibata.shuhei.746@s.kyushu-u.ac.jp}\\
   2020 Mathematics Subject Classification. 60F17; 60K35; 05C81.}}
\date{}
\begin{document}
\maketitle
\begin{abstract}
This paper investigates functional limit theorems for the Elephant Random Walk (ERW) on general periodic structures, extending the Bertenghi's results on $\mathbb{Z}^d$. Our results reveal new structure-dependent quantities that do not appear in the classical setting $\mathbb{Z}^d$, highlighting how the underlying structure affects the asymptotic behavior of the walk.
\end{abstract}

\section{Introduction}

The asymptotic behavior of random walks with long-range memory has been extensively studied in recent years. In particular, Elephant Random Walk (ERW) has attracted considerable attention. The ERW model was introduced by Sch\"{u}tz and Trimper \cite{schutz2004elephants} in 2004 to study memory effects in a one-dimensional discrete-time random walk with a complete memory of its past, and it exhibits a phase transition from diffusive to superdiffusive behavior.

One natural generalization of the one-dimensional ERW is to consider higher-dimensional settings, particularly random walks on the standard lattice $\mathbb{Z}^d$. A number of works have addressed this extension (cf.\cite{bercu2019martierw, bercu2021onthecenter, bertenghi2022functionallimit, guerin2023fixed, gonzalez2020multidimensional, curien2023recurrence, bercu2025multierwstop, chen2023analysis, qin2025reccurence}).
Bercu and Laulin \cite{bercu2019martierw} investigated the asymptotic behavior of the Multi-dimensional Elephant Random Walk (MERW) via a martingale approach. Later, Bertenghi \cite{bertenghi2022functionallimit} established functional limit theorems for MERW by employing a P\'{o}lya-type urn analysis, following the approach initiated by Baur and Bertoin \cite{baur2016elephant}. This connection between the (M)ERW and P\'{o}lya urns enables the derivation of functional limit theorems for the (M)ERW within the general framework developed by Janson \cite{janson2004functional}.

In addition to studying the asymptotic behavior of a single walker, another natural extension is to investigate the interaction of multiple ERWs, particularly the collision problem—that is, whether two independent elephant random walks on the same lattice meet infinitely often or only finitely many times. Roy, Takei and Tanemura \cite{roy2024often} studied the case where both ERWs on $\mathbb{Z}$ have the same memory parameter. Later, Shibata and Shirai \cite{shibata2025remark} extended their results to the case of different memory parameters and also obtained asymptotic results for the distance between them. Although this paper does not address the collision problem, we mention it here to emphasize the diversity of recent developments in the study of ERW.

Most of the works mentioned above focus on the ERWs on the standard lattice $\mathbb{Z}^d$.
In this paper, we extend this framework to more general structures, including $\mathbb{Z}^d$, as well as the \textit{triangular lattice}, the \textit{hexagonal lattice} and the \textit{brick wall}, see Fig.~\ref{fig:lattices}. Inspired by the work of Bertenghi \cite{bertenghi2022functionallimit}, who established functional limit theorems for the ERW on $\mathbb{Z}^d$, we derive the corresponding asymptotic results for the ERW on more general state spaces such as these lattices, using a P\'{o}lya-type urn approach. Our results reveal new underlying structure-dependent quantities which are not observed in the classical $\mathbb{Z}^d$ case, demonstrating that the underlying structure significantly influences the asymptotic behavior of the walk.

\begin{figure}[htbp]
    \centering
    \begin{subfigure}{0.3\textwidth}
        \centering
        \includegraphics[width=\linewidth]{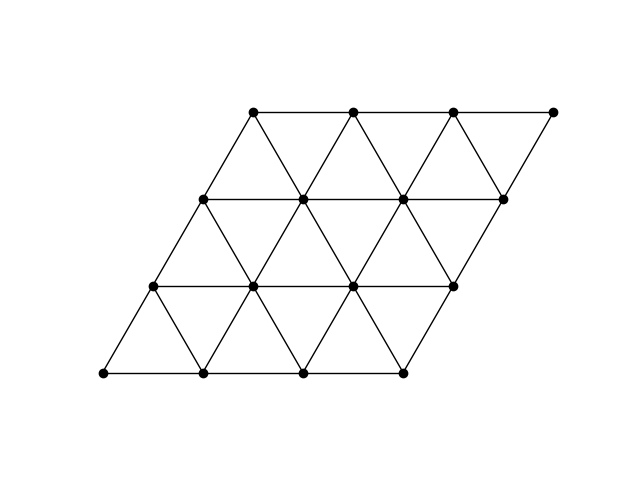}
        \label{fig:triangular}
    \end{subfigure}
    \begin{subfigure}{0.3\textwidth}
        \centering
        \includegraphics[width=\linewidth]{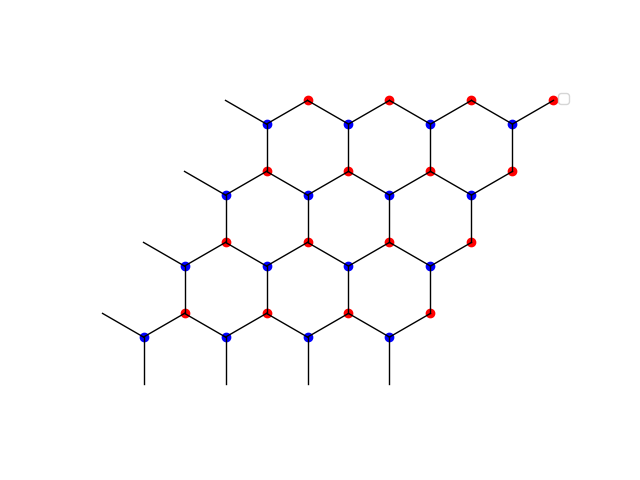}
        \label{fig:hexagonal}
    \end{subfigure}
    \begin{subfigure}{0.3\textwidth}
        \centering
        \includegraphics[width=\linewidth]{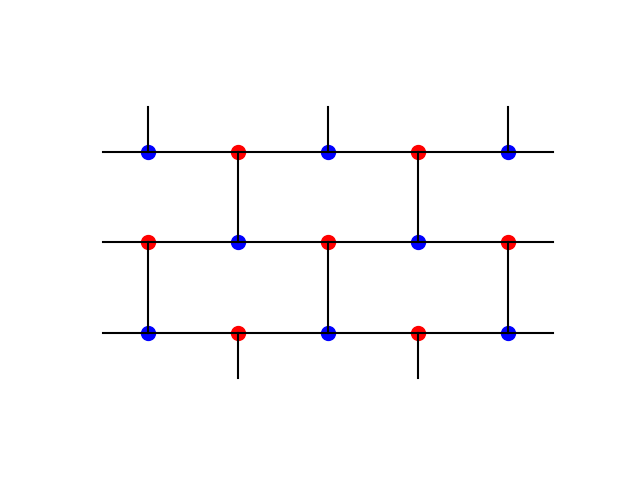}
        \label{fig:blockfence}
    \end{subfigure}
      \caption{\textbf{Figure $1$}. Left: Triangular lattice. All vertices (black dots) are structurally equivalent. Center: Hexagonal lattice. The vertices are partitioned into two structurally distinct classes, represented by red and blue dots. Right: Brick wall. As in the hexagonal lattice, the vertices are divided into two structurally distinct classes (red and blue dots).}
     \label{fig:lattices}
\end{figure}

The remainder of the paper is organized as follows. In Section $2$ and Section $3$, we provide the mathematical definition of the state spaces we consider and the ERW on them. In Section $4$, we introduce  discrete-time urn models that serves as a framework for describing the position of the ERW. In Section $5$, we present the main asymptotic results for ERW based on the urn analysis. Finally, in Section $6$, for the reader's convenience, we summarize the key quantities derived in this paper for several typical examples.

\section{Settings}\label{setting main}

To set the stage for our results, we first introduce the mathematical framework considered in this paper.

Given two finite sets of distinct vectors $\mathcal{U}:=\{u_1,\dots,u_m\}\subset \mathbb{R}^d$ and $\mathcal{W}:=\{w_1,\dots,w_{m'}\}\subset \mathbb{R}^d$ for fixed integers $m, m'\ge 2$. We call the elements of $\mathcal{U}$ and $\mathcal{W}$ \textit{step vectors} since, in Section~\ref{ERW def V_0=2}, we will define a random walk whose steps are drawn from these sets. We will consider two settings depending on whether the step sets coincide, namely the case $\mathcal{U}=\mathcal{W}$ and $\mathcal{U}\neq \mathcal{W}$.

\begin{itemize}
    \item \textbf{Case $\mathcal{U}=\mathcal{W}$:} In this case,
we consider the semigroup generated by them, that is, 
\begin{equation}\label{setting 1,2}
     \Gamma=\{\sum_{i=1}^m k_i u_i;\ k_i\in \mathbb{N}\cup \{0\}\}\subset \mathbb{R}^d.
 \end{equation}
Here, we assume that this semigroup $\Gamma$ is a \textit{lattice} in $\mathbb{R}^d$, i.e., it is a discrete additive subgroup that spans $\mathbb{R}^d$ as a vector space over $\mathbb{R}$. For this case, $\Gamma$ is analyzed as a single-colored vertex set, as in the triangular lattice on the left in Fig.~\ref{fig:lattices}, owing to the structural indistinguishability of all vertices.

\item \textbf{Case $\mathcal{U}\neq \mathcal{W}$:}
In this case, we first define the semigroup
\begin{equation*}
    \Gamma_0=\{\sum_{i=1}^m k_i u_i+\sum_{j=1}^{m'} l_j w_j;\ \sum_{i=1}^m k_i=\sum_{j=1}^{m'}l_j,\ k_i,l_j\in \mathbb{N}\cup \{0\}\}.
\end{equation*}
Here, we assume that this semigroup is a lattice in $\mathbb{R}^d$. We then consider the following state space $\Gamma$ that can be expressed as the disjoint union of $\Gamma_0$ and the translated set $\Gamma_0+\mathcal{U}$:
\begin{equation}\label{setting 2,2}
    \Gamma=\Gamma_0\sqcup (\Gamma_0 +\mathcal{U})\subset \mathbb{R}^d.
\end{equation}
We define the vertex classes by $Z_\mathcal{U}:=\Gamma_0$ and $Z_\mathcal{W}:=\Gamma_0+\mathcal{U}$. Under these settings, we enforce the \textit{Alternating Rule}: 
\begin{itemize}
    \item For any $z\in Z_\mathcal{U}$ and $i=1,\dots,m$, $z+u_i\in Z_\mathcal{W}$.
    \item For any $z\in Z_\mathcal{W}$ and $j=1,\dots,m'$, $z+w_j\in Z_\mathcal{U}$.
\end{itemize}
This rule ensures that the walk defined in Section~\ref{ERW def V_0=2} always remains within $\Gamma$. $\Gamma$ is therefore a bipartite graph with vertices classified according to the decomposition in (\ref{setting 2,2}), and is analyzed as a two-colored vertex set, as in the hexagonal lattice and brick wall shown in the center and right panels of Fig.~\ref{fig:lattices}.
\end{itemize}

\begin{rem}
When $\mathcal{U}=\mathcal{W}= \{\pm e_1, \pm e_2, \dots,\pm e_d\}$ where $e_i$ denotes the $i$-th standard basis vector of $\mathbb{R}^d$, $\Gamma$ coincides with the standard integer lattice $\mathbb{Z}^d$. Note that there are also many other ways to choose the vectors that generate $\mathbb{Z}^d$.
Various well-known structures, such as the triangular lattice, the hexagonal lattice and the brick wall, see Fig.~\ref{fig:lattices}, can be realized as $\Gamma$, by choosing suitable families of step vectors $\mathcal{U}$ and $\mathcal{W}$. For clarity, in Section~\ref{section examples}, we list the representative choices of $\mathcal{U}$ and $\mathcal{W}$ for several typical examples. 
\end{rem}

\begin{rem}
We introduce the $\Gamma$ given by (\ref{setting 2,2}) to represent a setting in which each vertex is assigned a specific type (or color), as in the hexagonal lattice.
Although the following example does not satisfy the decomposition of (\ref{setting 2,2}), one could construct an ERW model and a corresponding setting that covers it:
\begin{equation*}
    \mathcal{U} = \{\pm e_1, \pm e_2, (1,1)^\top\},\quad \mathcal{W} = \{\pm e_1, \pm e_2\}.
\end{equation*}
 
 We also note that the trivial case $m=1$ can be analyzed within the same framework, but it is excluded from our study.
\end{rem}

\section{Definition of the ERW on \texorpdfstring{$\Gamma$}{Gamma}}\label{ERW def V_0=2}
Next, we introduce the mathematical definitions of the ERW $\{S_n\}_{n=0}^{\infty}$ on $\Gamma$.  

For a given dimension $d\ge 1$ and the sets $\mathcal{U}$ and $\mathcal{W}$ introduced in Section~\ref{setting main}, let $\{\sigma_i\}_{i=1}^\infty$ and $\{\tau_j\}_{j=1}^\infty$ be random sequences taking values in $\mathcal{U}$ and $\mathcal{W}$, respectively. First, we discuss the case $\mathcal{U}\neq\mathcal{W}$. We consider an ERW that alternately moves between two vertex classes $Z_\mathcal{U}$ and $Z_\mathcal{W}$ by making a step chosen from $\mathcal{U}$ when at $Z_{\mathcal{U}}$, and a step chosen from $\mathcal{W}$ when at $Z_{\mathcal{W}}$. At time $n=0$, the elephant is started from the origin in $Z_{\mathcal{U}}$. Let $1\le i_0\le m$ and $1\le j_0\le m'$ be fixed. At time $n=1$, for the sake of definiteness, the elephant takes a deterministic step in the direction of $u_{i_0}$, i.e., $\sigma_1=u_{i_0}$. By the Alternating Rule, it then moves from the origin to a vertex in $Z_{\mathcal{W}}$. At time $n=2$, the elephant takes a deterministic step in the direction of $w_{j_0}$, i.e., $\tau_1=w_{j_0}$. By the Alternating Rule, it then moves from the vertex in $Z_{\mathcal{W}}$ to a vertex in $Z_{\mathcal{U}}$. Subsequently, the elephant proceeds by alternately visiting the two vertex classes $Z_{\mathcal{U}}$ and $Z_{\mathcal{W}}$, 
and generates an alternating sequence of step choices, $\sigma_1,\tau_1,\sigma_2,\tau_2,\sigma_3,\tau_3,\dots$.

For each $n=1,2,\dots$, given the past steps $\sigma_1, \sigma_2, \dots, \sigma_n$ and $\tau_1, \tau_2, \dots, \tau_n$, we define the $(n+1)$-th step of the two processes by
\begin{equation*}
    \P(\sigma_{n+1}=\sigma_{U_n})=p,
\end{equation*}
\begin{equation*}
    \P(\sigma_{n+1}=\sigma)=\frac{1-p}{m-1}\quad \text{for all $\sigma\in \mathcal{U}\setminus \{\sigma_{U_n}\}$},
\end{equation*}
and
\begin{equation*}
    \P(\tau_{n+1}=\tau_{W_n})=p,
\end{equation*}
\begin{equation*}
    \P(\tau_{n+1}=\tau)=\frac{1-p}{m'-1}\quad \text{for all $\tau\in \mathcal{W}\setminus \{\tau_{W_n}\}$},
\end{equation*}
where $p \in (0, 1)$ is called the \textit{memory parameter}, and the random variables $U_n$ and $W_n$ are uniformly distributed on $\{1,2,\dots,n\}$. We assume that all four collections of random variables, $U_n$ and $W_n$ and the process $\{\sigma_i\}_{i=1}^n$ and the process $\{\tau_j\}_{j=1}^n$, are mutually independent. 

The sequences $\{\sigma_i\}_{i=1}^\infty$ and $\{\tau_j\}_{j=1}^\infty$ generate the ERW $\{S_n\}_{n=0}^\infty$ on $\Gamma$ given by (\ref{setting 2,2}) with $S_0=\mathbf{0}$ defined by
\begin{equation}\label{def ERW S 2}
    S_{2n}=\sum_{i=1}^{n} (\sigma_i+\tau_i) \quad \text{and} \quad S_{2n-1}=S_{2(n-1)}+\sigma_n \quad \text{for $n=1,2,\dots$},
\end{equation}
where $\mathbf{0}:=(0,\dots,0)^\top\in \mathbb{R}^d$. We call this process \textit{Type-II ERW}. It is important to note that the ERW moves alternately between $Z_\mathcal{U}$ and $Z_\mathcal{W}$ with taking values alternately from  $\{\sigma_i\}_{i=1}^\infty$ and $\{\tau_j\}_{j=1}^\infty$. This ensures that over any $2n$ steps, the number of movements drawn from the set $\mathcal{U}$ is precisely $n$, and similarly, the number of movements drawn from $\mathcal{W}$ is precisely $n$. This deterministic balance in step counts plays a key role in our analysis.

Next, we discuss the case $\mathcal{U}=\mathcal{W}$. We note that since the process depends on the entire past trajectory, we describe the definition separately from the case $\mathcal{U}\neq\mathcal{W}$.

The ERW $\{S_n\}_{n=0}^\infty$ on $\Gamma$ given by (\ref{setting 1,2}) with $S_0=\mathbf{0}$ is defined by 
    \begin{equation}\label{def ERW S 1}
      S_n=\sum_{i=1}^n \sigma_i  \quad \text{for $n=1,2,\dots$}.
    \end{equation} 
    We call this process \textit{Type-I ERW}. When $\mathcal{U}=\mathcal{W}= \{\pm e_1, \pm e_2, \dots,\pm e_d\}$, i.e., $\Gamma=\mathbb{Z}^d$, we recover the MERW introduced by Bercu and Laulin \cite{bercu2019martierw} and Bertenghi \cite{bertenghi2022functionallimit}.

\begin{rem}\label{rem U=W case}
    The essential difference between Type-I ERW and Type-II ERW does not lie in the distinction between $\mathcal{U}=\mathcal{W}$ and  $\mathcal{U}\neq \mathcal{W}$, but rather in how past steps are chosen: For Type-II ERW, only steps chosen from odd or even times are used, whereas for Type-I ERW, the choice depends on the entire past.
For example, when considering $\mathcal{U}=\mathcal{W}=\{\pm e_1, \pm e_2\}$, our current formulation defines the ERW using (\ref{def ERW S 1}), but one could also consider an ERW based on (\ref{def ERW S 2}) and then this yields a different stochastic model. However, in Section~\ref{section main result 2}, we show that the limiting distributions of Type-II ERW under  $\mathcal{U}=\mathcal{W}$ and Type-I ERW coincide, except for the superdiffusive case where the initial step affects the asymptotic behavior of the walk, see Theorem~\ref{supersuperdiffusive V_0=2}. 
\end{rem}

\section{Connection to the P\'{o}lya-type urns}\label{Section urn}

In this section, we explain the relation between the ERW on  $\Gamma$ and P\'{o}lya-type urn models. The urn model used for the ERW belongs to the results provided by Janson \cite{janson2004functional} for the more general framework of P\'{o}lya urns. 

First, we discuss the case  $\mathcal{U}\neq\mathcal{W}$. Consider two discrete-time urns. One urn contains balls of $m$ distinct colors, and the other contains balls of $m'$ distinct colors. At any time $n\in \N$, the compositions of the two urns are described by $m$-dimensional random vector $X_n=(X_n^1,X_n^2,\dots,X_n^m)^\top\in (\N\cup \{0\})^m$ and $m'$-dimensional random vector $Y_n=(Y_n^1,Y_n^2,\dots,Y_n^{m'})^\top\in (\N\cup \{0\})^{m'}$, where each component $X_n^i$ and $Y_n^j$ denote the number of balls of color $i\in \{1,\dots,m\}$ and $j\in \{1,\dots,m'\}$ at time $n$, respectively. Let $X_0=\mathbf{0}$ and $Y_0=\mathbf{0}$, that is there are no balls in the two urns at time zero. Let $1\le i_0\le m$ and $1\le j_0\le m'$ be as defined in Section \ref{ERW def V_0=2}. At time $n=1$, we restrict ourselves to the initial deterministic configuration $X_1=\mathbf{e}_{i_0}$ and $Y_1=\mathbf{e}^{'}_{j_0}$, where $\mathbf{e}_k$ and $\mathbf{e}^{'}_k$ denote the $k$-th standard basis vectors of $\mathbb{Z}^m$ and $\mathbb{Z}^{m'}$, respectively. For each urn process at any time $n\ge 2$, we draw one ball uniformly at random from the urn, observe its color, return it to the same urn, and then, with probability $p$, add one additional ball of the same color to the same urn. With the remaining probability $1 - p$, we instead add one ball among the $m - 1$ remaining colors to the same urn, each with probability $(1 - p)/(m - 1)$. We then update $X_n=(X_n^1,X_n^2,\dots,X_n^m)^\top$ accordingly, so that $X_n$ describes the composition of the urn after the $(n - 1)$-th drawing. The process $Y_n$ is defined similarly to $X_n$, and the two processes are assumed to be independent. 

The connection with the ERW model is the following: If $\{S_n\}_{n=0}^{\infty}$ is Type-II ERW, then we have 
\begin{align}\label{relation between ERW and urn with |V_0|=2|}
    \{S_{2n}\}_{n=0}^{\infty}&\inlaw \{\sum_{i=1}^m X_n^iu_i+\sum_{j=1}^{m'}Y_n^{j}w_j\}_{n=0}^\infty,
\end{align} 
where $\inlaw$ denotes equality in distribution. This means that the extent to which the ERW moves in the directions $u_i$ and $w_j$ at time $n$ is governed by the numbers of balls $X_n^i$ and $Y_n^j$, respectively.  

In the case $\mathcal{U}=\mathcal{W}$, the connection with Type-I ERW is given by
    \begin{align}\label{relation between ERW and urn with |V_0|=1|}
    \{S_{2n}\}_{n=0}^{\infty}&\inlaw \{\sum_{i=1}^m X_{2n}^iu_i\}_{n=0}^\infty.
\end{align} 
\begin{rem}\label{rem U=W polya connection}
    We note that in general, even when $\mathcal{U}=\mathcal{W}$, (\ref{relation between ERW and urn with |V_0|=2|}) and (\ref{relation between ERW and urn with |V_0|=1|}) cannot be combined, as sampling from the entire urn process $X^i_{2n}$ (whole past) and sampling only from one of the two urn processes $X^i_n$ and $Y^j_n$ (odd or even times) produce distinct stochastic processes, see also Remark~\ref{rem U=W case}.
\end{rem}

As demonstrated in (\ref{relation between ERW and urn with |V_0|=2|}) and (\ref{relation between ERW and urn with |V_0|=1|}), to understand the long-time behavior of the ERW, it suffices to study the asymptotic behavior of the corresponding urn process. The key quantities governing the asymptotics of the urn process are the eigenvalues and eigenvectors of the so-called \textit{replacement matrix}. In our case, the replacement matrices associated with the processes $X_n$ and $Y_n$ are given by 
 the $m\times m$ matrix
\begin{equation}\label{replacement matrix 2}
        A=\begin{pmatrix}
            p &\frac{1-p}{m-1}&\cdots &\frac{1-p}{m-1}\\
            \frac{1-p}{m-1} &p&\ddots&\vdots\\
            \vdots&\ddots&\ddots&\vdots\\
            \frac{1-p}{m-1}&\cdots &\cdots &p
        \end{pmatrix}
        =\frac{1-p}{m-1}J_m +\frac{mp-1}{m-1}I_m,
    \end{equation}
and the $m'\times m'$ matrix
\begin{equation}\label{replacement matrix 2'}
        A'=\begin{pmatrix}
            p &\frac{1-p}{m'-1}&\cdots &\frac{1-p}{m'-1}\\
            \frac{1-p}{m'-1} &p&\ddots&\vdots\\
            \vdots&\ddots&\ddots&\vdots\\
            \frac{1-p}{m'-1}&\cdots &\cdots &p
        \end{pmatrix}
        =\frac{1-p}{m'-1}J_{m'} +\frac{m'p-1}{m'-1}I_{m'},
    \end{equation}
    where $J_m$ denotes $m\times m$ all-ones matrix and $I_m$ is the $m\times m$ identity matrix. The matrix $A$ has the eigenvalues $\lambda_1 = 1$ of multiplicity one and $\lambda_2 = (mp - 1)/(m - 1)$ of multiplicity $m - 1$. The right eigenvector associated with the largest eigenvalue $\lambda_1$ is $v_1=\frac{1}{m}(1,\dots,1)^\top\in \R^m$. The same holds for $J_{m'}, I_{m'},\lambda_1^{'},\lambda_2^{'}$ and $v_1^{'}$ substituting $m$ with $m'$.

It is well-known (cf. \cite{athreya1968embedding, chauvin2011limit, janson2004functional, kesten1966additional}) that the asymptotic behavior  of the urn process depends on the position of $\lambda_2/\lambda_1$ relative to $1/2$. In our settings, for each process, the critical values associated with the memory parameter $p$ are
\begin{equation}\label{critical parameter 2}
    p_c^m=\frac{m+1}{2m}
\end{equation}
by solving $\lambda_2/\lambda_1= 1/2 $, and
\begin{equation}\label{critical parameter 2'}
    p_c^{m'}=\frac{m'+1}{2m'}
\end{equation}
by solving $\lambda_2^{'}/\lambda_1^{'}= 1/2 $. 
We observe that when $\mathcal{U}=\mathcal{W}= \{\pm e_1, \pm e_2, \dots,\pm e_d\}$, i.e., $m=2d$, then $p_c^m=(2d+1)/(4d)$, which coincides with the critical value derived by Bercu and Laulin \cite{bercu2019martierw} and Bertenghi \cite{bertenghi2022functionallimit}. Moreover, for $d = 1$, we recover $p_c^m = 3/4$, corresponding to the phase transition originally observed by Sch\"{u}tz and Trimper \cite{schutz2004elephants}. 

As shown in (\ref{relation between ERW and urn with |V_0|=2|}), when $\mathcal{U}\neq \mathcal{W}$, we need to consider two P\'{o}lya urn processes within a single ERW. In Section \ref{section main result 2}, we derive asymptotic results for ERW by distinguishing cases according to the relative magnitudes of $ p_c^m$ and $p_c^{m'}$.

In addition, defining $a=(mp-1)/(m-1)$, we note that the following equivalence is obtained:
\begin{equation}\label{superdiffusive relation a and p}
    a<\frac{1}{2}\Leftrightarrow p<p_c^m,\quad a=\frac{1}{2}\Leftrightarrow p=p_c^m,\quad a>\frac{1}{2}\Leftrightarrow p>p_c^m.
\end{equation}
This relation also holds for $p_c^{m'}$ and $a'$, by taking $a'=(m'p-1)/(m'-1)$. 

\begin{rem}
Obviously, 
we can generalize the case where $\Gamma$ is generated by the sets of vectors $\{u_{i,1}\}_{i=1}^{m_1}, \dots,\{u_{i,l}\}_{i=1}^{m_l}\subset \mathbb{R}^d$ for $l\ge3$. The ERW $\{S_n\}_{n=0}^\infty$ is defined on the multipartite graph $\Gamma$ partitioned into $l$ vertex classes, $Z_1, Z_2, \dots, Z_l$. The walk is enforced to move cyclically between vertex classes: A step taken from $Z_i$ must lead to a vertex in $Z_{i+1}$ for $i=1,\dots,l-1$, and a step taken from $Z_l$ must lead back to a vertex in $Z_1$. Under this procedure, we can establish the same class of results for the ERW as those presented in Section \ref{section main result 2} by employing similar analytical techniques. However, considering the \textit{kagome lattice}, see Fig.~\ref{fig:kagome lattice}, corresponding to $l=3$, the walk on such a multipartite graph moves randomly in general, and unlike the above cyclical setting or the strictly bipartite case $l=2$ considered in this paper, the number of visits to each vertex class is not deterministically balanced and a different type of analysis is required.
\begin{figure}[htbp]
    \centering
    \includegraphics[width=0.4\linewidth]{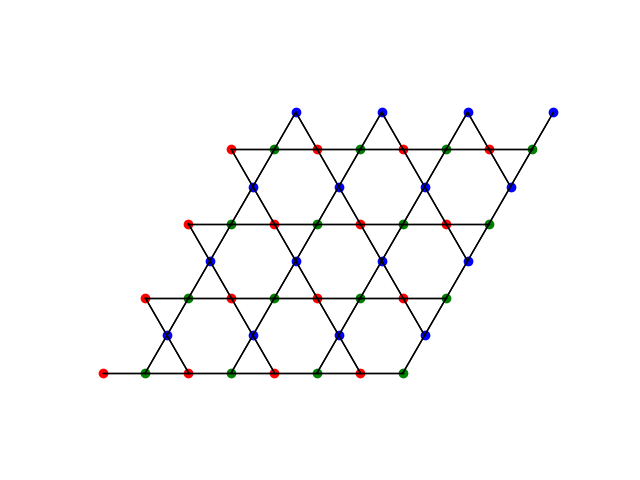}
    \caption{\textbf{Figure $2$}. Kagome lattice. The vertices are partitioned into three structurally distinct classes, represented by red, blue, and green dots.}
    \label{fig:kagome lattice}
\end{figure}  
    \end{rem}

\section{Main result}\label{section main result 2}

In this section, we derive a strong law of large numbers and functional limit theorems for the ERW on $\Gamma$ in all regimes. For our purpose, it is most convenient to adapt the general results of Janson \cite{janson2004functional} and to translate them into the setting of the ERW by using the connection established in (\ref{relation between ERW and urn with |V_0|=2|}) and (\ref{relation between ERW and urn with |V_0|=1|}) together with the continuous mapping theorem. 

It should be noted that the results presented below extend those results found in Bertenghi \cite{bertenghi2022functionallimit}. We also show, in relation to Remark~\ref{rem U=W case} and Remark~\ref{rem U=W polya connection}, that the limiting distributions of Type-II ERW under  $\mathcal{U}=\mathcal{W}$ and Type-I ERW coincide,  except for the superdiffusive regime. 

\subsection{A LLN-type convergence result}

Our first result concerns a strong law of large numbers.  The following property holds regardless of the regime, in other words, we eliminate the dependency on the numbers $m$ and $m'$ of edges emanating from each vertex.
\begin{thm}\label{LLN}
Let $\{S_n\}_{n=0}^\infty$ be Type-I or Type-II ERW. Then, for all $p\in(0,1)$, we have the almost sure convergence
\begin{equation*}
    \frac{S_{n}}{n}\to \frac{1}{2}(\Bar{u}+\Bar{w}) \quad \text{as $n\to \infty$},
    \end{equation*}
    where $\Bar{u}=\frac{1}{m}\sum_{i=1}^m u_i$ and $\Bar{w}=\frac{1}{m'}\sum_{j=1}^{m'} w_j$.
\end{thm}
\begin{proof}
First, we consider Type-II ERW. By Theorem $3.21$ in Janson \cite{janson2004functional}, we have the almost sure convergence
\begin{equation*}
    \frac{X_n}{n}\to \lambda_1 v_1=\frac{1}{m}\begin{pmatrix}
    1 \\ 1 \\ \vdots \\ 1 \end{pmatrix} \quad \text{as $n\to \infty$}
\end{equation*}
and
\begin{equation*}
    \frac{Y_n}{n}\to \lambda_1^{'} v_1^{'}=\frac{1}{m'}\begin{pmatrix}
    1 \\ 1 \\ \vdots \\ 1 \end{pmatrix} \quad \text{as $n\to \infty$},
\end{equation*}
where $\lambda_1$ and $\lambda_1^{'}$ are the largest eigenvalue of the replacement matrix $A$ and $A'$ as given in (\ref{replacement matrix 2}) and (\ref{replacement matrix 2'}), and  $v_1$ and $v_1^{'}$ are its corresponding (right) eigenvector, respectively. Hence, the claim follows by (\ref{relation between ERW and urn with |V_0|=2|}) and the continuous mapping theorem.

Similarly, for Type-I ERW, we obtain the almost sure convergence $S_{n}/n\to \Bar{u}$ as $n\to \infty$ from (\ref{relation between ERW and urn with |V_0|=1|}). 

This concludes the proof.
\end{proof}

\begin{rem}\label{rem p=1}
If $p=1$, which is not covered by the law of large numbers of Theorem~\ref{LLN}, the
ERW is trivial since by definition of the process one then has $\sigma_n=\sigma_1$ and $\tau_n=\tau_1$ for all $n \ge 1$. Hence,
$S_{2n}/n=u_{i_0}+w_{j_0}$ and $S_{n}/n=u_{i_0}$ hold true for Type-II and Type-I ERW, respectively.
\end{rem}

\subsection{Functional-type convergence results}

Next, we discuss the functional-type limit theorems. Their convergence results concern the distributional convergence of the ERW on $\Gamma$ in the Skorokhod space $D([0,\infty))$
of right-continuous functions with left-hand limits. In what follows, we assume without loss of generality that $p_c^m\le p_c^{m'}$, as the opposite case can be treated similarly. We classify the possible situations into three cases and we derive the corresponding asymptotic results: (I) $0< p<p_c^m\le p_c^{m'}$ (Theorem~\ref{0<p<p_c^m p_c^{m'}}), (II) $0< p_c^m=p=p_c^{m'}$ (Theorem~\ref{0<p_c^m=p=p_c^{m'}}), 
(III) $0<p_c^m= p_c^{m'}<p$ (Theorem~\ref{supersuperdiffusive V_0=2}). 
The remaining parameter regimes 
$0 < p_c^m = p < p_c^{m'}$, 
$0 < p_c^m < p \le p_c^{m'}$, 
and 
$0 < p_c^m < p_c^{m'} < p$
are discussed in 
Remark~\ref{rem 0<p_c^m=p<p_c^{m'}} 
and 
Remark~\ref{rem 0<p_c^m< p_c^m'<p}.

First, we provide the functional central limit theorem in the diffusive regime.

\begin{thm}\label{0<p<p_c^m p_c^{m'}}(\textbf{\textit{diffusive regime}}) Let $\{S_n\}_{n=0}^\infty$ be Type-I or Type-II ERW and $0< p<p_c^m\le p_c^{m'}$. Then, we have the distributional convergence in $D[0,\infty)$,
\begin{equation*}
    \left\{\frac{S_{\lfloor 2nt\rfloor}-nt(\Bar{u}+\Bar{w})}{\sqrt{n}}\right\}_{t\ge 0} \Rightarrow \left\{W_t\right\}_{t\ge 0} \quad \text{as $n\to \infty$},
\end{equation*}
where $\left\{W_t\right\}_{t\ge 0}$ is a centered continuous $\R^d$-valued Gaussian process with $W_0=\mathbf{0}$, and for $0<s\le t$, the covariance structure is specified by
\begin{equation*}
    \mathbb{E}W_sW_t^\top=C_{a} s\left(\frac{t}{s}\right)^{a}\Sigma(\mathcal{U})+C_{a'}s\left(\frac{t}{s}\right)^{a'}\Sigma(\mathcal{W}),
\end{equation*}
where $a$ is given in (\ref{superdiffusive relation a and p}), $C_{a}=1/(1-2a)$,
\begin{equation*}
\Sigma(\mathcal{U})=\frac{1}{m}\sum_{i=1}^m \left(u_i-\Bar{u}\right)\left(u_i-\Bar{u}\right)^\top\quad \text{and}\quad \Sigma(\mathcal{W})=\frac{1}{m'}\sum_{j=1}^{m'} \left(w_j-\Bar{w}\right)\left(w_j-\Bar{w}\right)^\top.
\end{equation*}
\end{thm}

\begin{rem}\label{rem diffusive V_0=1}
In Bertenghi case i.e.,  $\mathcal{U}=\mathcal{W}= \{\pm e_1, \pm e_2, \dots,\pm e_d\}$, we have $\Bar{u}=\mathbf{0}$, and $\Sigma(\mathcal{U})$ deduces to $I_d/d$ where $I_d$ is the $d\times d$ identity matrix. Therefore, the covariance coincides with the form given by Theorem $4.2$ in \cite{bertenghi2022functionallimit}. Furthermore, in this case, the coordinates of the limiting process $\{W_t=(W^1_t,\dots,W^d_t)^\top\}_{t\ge 0}$ can be expressed as independent \textit{noise reinforced Brownian motion} (cf. \cite{bertenghi2022functionallimit,bertoin2020noise, bertoin2021universality}). 
  
In general, $\Sigma(\mathcal{U})$ and $\Sigma(\mathcal{W})$ are typically non-diagonal, indicating that the diffusion along different coordinate axes is correlated. In this case, $W_t$ can be represented by the following Wiener integral: 
\begin{equation*}
      W_t= \Sigma(\mathcal{U})^{\frac{1}{2}}\int_0^t \left(\frac{t}{v}\right)^adB_v+\Sigma(\mathcal{W})^{\frac{1}{2}}\int_0^t \left(\frac{t}{v}\right)^{a'}dB^{'}_v,
\end{equation*}
where $\{B_t\}_{t\ge 0}$ and  $\{B^{'}_t\}_{t\ge 0}$ are independent standard $d$-dimensional Brownian motions. Moreover, as shown by Lamperti \cite{lamperti1962semi}, the limiting process $\{W_t\}_{t\ge 0}$ appearing in such a functional limit theorem is known to be a \textit{self-similar process}.
\end{rem}

\begin{proof}[Proof of Theorem~\ref{0<p<p_c^m p_c^{m'}}]
First, we consider Type-II ERW. We apply to Theorem $3.31$(i) in Janson \cite{janson2004functional}, which establishes
that $\{n^{-1/2}(X_{\lfloor nt\rfloor}-nt\lambda_1v_1)\}_{t\ge 0}$ and $\{n^{-1/2}(Y_{\lfloor nt\rfloor}-nt\lambda^{'}_1v^{'}_1)\}_{t\ge 0}$ converge in distribution in $D([0,\infty))$ towards a centered continuous $\R^m$-valued Gaussian processes $\{V_t\}_{t\ge 0}$ and a centered continuous $\R^{m'}$-valued Gaussian processes $\{V_t^{'}\}_{t\ge 0}$ with $V_0=V_0^{'}=\mathbf{0}$, respectively. We note that $\{V_t\}_{t\ge 0}$ and  $\{V_t^{'}\}_{t\ge 0}$ are independent, and the covariance structure of $\{V_t\}_{t\ge 0}$ is closer
specified under Remark $5.7$ Display $(5.6)$ in Janson \cite{janson2004functional}. By technical calculation for its covariance (see, Appendix A in Bertenghi \cite{bertenghi2022functionallimit} for details), we obtain for $0<s\le t$,
\begin{equation}\label{cov urn V_0=2 2}
    \E[V_s V_t^\top]=\frac{c_{a,s,t}}{m} \left(I_m-\frac{J_m}{m}\right)
\end{equation}
and 
\begin{equation}\label{cov urn V_0=2 2'}
    \E[V_s^{'} (V_t^{'})^\top]= \frac{c_{a',s,t}}{m'} \left(I_{m'}-\frac{J_{m^{'}}}{m'}\right),
\end{equation}
where  
\begin{equation}\label{coefficient C}
    c_{a,s,t}=C_{a}s\left(\frac{t}{s}\right)^a.
\end{equation}

By (\ref{relation between ERW and urn with |V_0|=2|}), we deduce that $\{n^{-1/2}(S_{\lfloor 2nt\rfloor}-(nt\Bar{u}+nt\Bar{w})\}_{t\ge 0}$ converges in distribution in $D([0,\infty))$ towards a $\R^d$-valued process, denoted by
$\{W_t\}_{t \ge 0}$ and specified by 
\begin{align}\label{gauss r.v. critical}
    W_t &=\sum_{i=1}^{m} V^{i}_tu_i +\sum_{j=1}^{m'}(V_t^{'})^{j}w_j
\end{align}
almost surely, where $V^i_t$ and $(V_t^{'})^{j}$ denote the $i$-th component of $V_t$ for $i=1,\dots, m$ and the $j$-th component of $V_t^{'}$ for $j=1,\dots, m'$, respectively. As $\{W_t\}_{t \ge 0}$ is a linear combination of the components of  $\{V_t\}_{t\ge 0}$ and  $\{V_t^{'}\}_{t\ge 0}$, it follows that $\{W_t\}_{t \ge 0}$ is a centered continuous $\mathbb{R}^d$-valued Gaussian process started from zero, and thus the law of $\{W_t\}_{t \ge 0}$ is completely determined by its covariance structure.

Since we have 
\begin{equation*}
    \mathbb{E}[V_s^iV_t^j]=\frac{c_{a,s,t}}{m}\left(\delta_{i,j}-\frac{1}{m}\right),
\end{equation*}  
where $\delta_{i,j}$ is the Kronecker delta function, it follows from the independence of $\{V_t\}_{t\ge 0}$ and  $\{V_t^{'}\}_{t\ge 0}$ that 
\begin{align*}
\E[W_s W_t^\top]
&= \sum_{i=1}^m\sum_{k=1}^m \E[V_s^iV_t^k]u_iu_k^\top+\sum_{j=1}^{m'} \sum_{l=1}^{m'}\E[(V_s^{'})^{j}(V_t^{'})^{l}]w_jw_l^\top\\
&=c_{a,s,t}\Sigma(\mathcal{U})+c_{a',s,t}\Sigma(\mathcal{W}).
\end{align*}

Similarly, for Type-I ERW, it follows from (\ref{relation between ERW and urn with |V_0|=1|}) and the continuity and self-similarity of $V_t$ that $\{n^{-1/2}(S_{\lfloor 2nt\rfloor}-2nt\Bar{u})\}_{t\ge 0}$ converges in distribution in $D([0,\infty))$ towards a centered continuous $\R^d$-valued Gaussian process $\{\Tilde{W_t}\}_{t \ge 0}$, and the covariance structure is specified by $\E[\Tilde{W}_s \Tilde{W_t}^\top]= 2c_{a,s,t}\Sigma(\mathcal{U})$. 

This concludes the proof.
\end{proof}

In the critical regime, we provide another functional central limit theorem with properly normalized.
\begin{thm}\label{0<p_c^m=p=p_c^{m'}}(\textbf{\textit{critical regime}}) Let $\{S_n\}_{n=0}^\infty$ be Type-I or Type-II ERW and $0<p_c^m=p=p_c^{m'}$. Then, we have the distibutional convergence in $D[0,\infty)$,
\begin{equation*}
    \left\{\frac{S_{\lfloor 2n^t\rfloor}-n^t(\Bar{u}+\Bar{w})}{n^{t/2}\sqrt{\log n}}\right\}_{t\ge 0} \Rightarrow \left\{W_t\right\}_{t\ge 0} \quad \text{as $n\to \infty$},
\end{equation*}
where $\left\{W_t\right\}_{t\ge 0}$ is a centered continuous $\R^d$-valued Gaussian process with $W_0=\mathbf{0}$, and for $0<s\le t$, the covariance structure is specified by
\begin{equation*}
\mathbb{E}W_sW_t^\top=s\Sigma(\mathcal{U},\mathcal{W}),
\end{equation*}
where 
\begin{equation}\label{Sigma_U,W}
\Sigma(\mathcal{U},\mathcal{W})=\Sigma(\mathcal{U})+\Sigma(\mathcal{W}).
\end{equation}
\end{thm}
\begin{rem}
Similarly to Remark~\ref{rem diffusive V_0=1}, when we consider the Bertenghi case, the covariance coincides with the form given by Theorem $4.3$ in \cite{bertenghi2022functionallimit}. Moreover, in this case, $\{B_t\}_{t\ge 0}=\{\sqrt{d}W_t\}_{t\ge 0}$ is a standard $d$-dimensional Brownian motion. 
    
In general, $W_t$ can be represented by $W_t=\Sigma(\mathcal{U},\mathcal{W})^{1/2}B_t$, where $\{B_t\}_{t\ge 0}$ denotes a standard $d$-dimensional Brownian motion.
\end{rem}

\begin{proof}[Proof of Theorem~\ref{0<p_c^m=p=p_c^{m'}}]
First, we consider Type-II ERW. We apply to Theorem $3.31$(ii) in Janson \cite{janson2004functional}, which establishes that $\{n^{-t/2}(\log n)^{-1/2}(X_{\lfloor n^t\rfloor}-n^t\lambda_1v_1)\}_{t\ge 0}$ and $\{n^{-t/2}(\log n)^{-1/2}(Y_{\lfloor n^t\rfloor}-n^t\lambda_1^{'}v_1^{'})\}_{t\ge 0}$ converge in distribution in $D([0,\infty))$ towards a centered continuous $\R^m$-valued Gaussian processes $\{V_t\}_{t\ge 0}$ and a centered continuous $\R^{m'}$-valued Gaussian processes $\{V_t^{'}\}_{t\ge 0}$ with $V_0=V^{'}_0=\mathbf{0}$, respectively. We note that $\{V_t\}_{t\ge 0}$ and  $\{V_t^{'}\}_{t\ge 0}$ are independent, and the covariance structures are given by Display $(3.27)$ in Janson \cite{janson2004functional}. By technical calculation for its covariance (see, Appendix A in Bertenghi \cite{bertenghi2022functionallimit} for details), we obtain for $0<s\le t$,
\begin{equation}\label{cov urn V_0=2, critical}
    \E[V_s V_t^\top]= \frac{s}{m} \left(I_m-\frac{J_m}{m}\right).
\end{equation}

By (\ref{relation between ERW and urn with |V_0|=2|}), we deduce that $\{n^{-t/2}(\log n)^{-1/2}(S_{\lfloor 2n^t\rfloor}-n^t(\Bar{u}+\Bar{w}))\}_{t\ge 0}$ converges in distribution in $D([0,\infty))$ towards an $\R^d$-valued process $\{W_t\}_{t \ge 0}$, specified by 
\begin{align}\label{gauss r.v. critical 2}
    W_t &=\sum_{i=1}^{m} V^{i}_tu_i +\sum_{j=1}^{m'}(V_t^{'})^{j}w_j.
\end{align}
almost surely, where $V^i_t$ and $(V_t^{'})^{j}$ denote the $i$-th component of $V_t$ for $i=1,\dots, m$ and the $j$-th component of $V_t^{'}$ for $j=1,\dots, m'$, respectively. As $\{W_t\}_{t \ge 0}$ is a linear combination of the components of  $\{V_t\}_{t\ge 0}$ and  $\{V_t^{'}\}_{t\ge 0}$, it follows that $\{W_t\}_{t \ge 0}$ is a centered continuous $\mathbb{R}^d$-valued Gaussian process started from zero, and its covariance structure follows from the independence of $\{V_t\}_{t\ge 0}$ and  $\{V_t^{'}\}_{t\ge 0}$ and (\ref{cov urn V_0=2, critical}) that
\begin{align*}
    \E[W_s W_t^\top]
    &=s\Sigma(\mathcal{U},\mathcal{W}).
\end{align*} 

Similarly, for Type-I ERW, it follows from (\ref{relation between ERW and urn with |V_0|=1|}) and the continuity of $V_t$ that $\{n^{-t/2}(\log n)^{-1/2}(S_{\lfloor 2n^t\rfloor}-2n^t\Bar{u})\}_{t\ge 0}$  converges in distribution in $D([0,\infty))$ towards a centered continuous $\R^d$-valued Gaussian process $\{\Tilde{W_t}\}_{t \ge 0}$ and the covariance structure is specified by $\E[\Tilde{W_s} \Tilde{W_t}^\top]= 2s\Sigma(\mathcal{U})$. 

This concludes the proof.
\end{proof}

\begin{rem}\label{rem 0<p_c^m=p<p_c^{m'}}
In the case $0<p_c^m=p<p_c^{m'}$, under the scaling by $n^{-t/2}(\log n)^{-1/2}$, we observe that the asymptotic behavior of $X$ dominates that of $Y$. Hence, for Type-II ERW, we obtain a centered continuous $\R^d$-valued Gaussian process $\left\{W_t\right\}_{t\ge 0}$ with the covariance structure 
\begin{equation*}
    \mathbb{E}W_sW_t^\top=s\Sigma(\mathcal{U}).
\end{equation*}
\end{rem}

Finally, we consider the limit theorems related to the superdiffusive regime. In contrast to the previous two regimes, the Type-II ERW under $\mathcal{U}=\mathcal{W}$ and the Type-I ERW exhibit different limiting distributions.

\begin{thm}\label{supersuperdiffusive V_0=2}(\textbf{\textit{superdiffusive regime}}) Let $0<p_c^m= p_c^{m'}<p$ and $a$ be given by (\ref{superdiffusive relation a and p}). Then, the following almost sure functional convergence holds: For Type-II ERW $\{S_n\}_{n=0}^\infty$, we have
\begin{equation*}
    \left\{\frac{S_{\lfloor 2nt\rfloor}-nt(\Bar{u}+\Bar{w})}{n^{a}}\right\}_{t\ge 0} \Rightarrow \left\{t^{a}L\right\}_{t\ge 0} \quad \text{as $n\to \infty$},
\end{equation*}
where $L$ is some $\mathbb{R}^d$-valued random variable different from zero.

For Type-I ERW, we have 
\begin{equation*}
    \left\{\frac{S_{\lfloor nt\rfloor}-nt\Bar{u}}{n^{a}}\right\}_{t\ge 0} \Rightarrow \left\{t^{a}\Tilde{L}\right\}_{t\ge 0} \quad \text{as $n\to \infty$},
\end{equation*}
where $\Tilde{L}$ is some $\mathbb{R}^d$-valued random variable different from zero.
\end{thm}
\begin{proof}
First, we consider Type-II ERW. We note that in the notation of Theorem $3.24$ in Janson \cite{janson2004functional}, we have $\Lambda'_
{III} = \{\lambda_2= a\}$. We are
therefore in the setting of the last part of the cited theorem and get that $\{n^{-a}(X_{\lfloor nt\rfloor}-nt\lambda_1v_1)\}_{t\ge 0}$ and $\{n^{-a}(Y_{\lfloor nt\rfloor}-nt\lambda_1v_1)\}_{t\ge 0}$ converge almost surely to $\{t^{a}\hat{W}_1\}_{t\ge 0}$ and $\{t^{a}\hat{W}_2\}_{t\ge 0}$, where $\hat{W}_1=(\hat{W}_1^1,\dots,\hat{W}_1^m)^\top$ and $\hat{W}_2=(\hat{W}_2^1,\dots,\hat{W}_2^m)^\top$ are some independent nonzero random vectors lying in the eigenspace $E_{\lambda_2}$ of the replacement matrix $A$ as displayed in (\ref{replacement matrix 2}). 

Hence, by  (\ref{relation between ERW and urn with |V_0|=2|}), we obtain the nonzero $\mathbb{R}^d$-valued random vector, which is constructed as the sum of independent random vectors,
\begin{equation}\label{superdiffusive r.v. V_0=2}
    L=\sum_{i=1}^m (\hat{W}_1^i u_i+\hat{W}_2^i w_i)
\end{equation}
almost surely.

Similarly, for Type-I ERW, we obtain the nonzero $\mathbb{R}^d$-valued random vector
\begin{equation}\label{superdiffusive r.v. V_0=1}
    \Tilde{L}=\sum_{i=1}^m \hat{W}_3^i u_i
\end{equation}
from (\ref{relation between ERW and urn with |V_0|=1|}) almost surely, where $\hat{W}_3=(\hat{W}_3^1,\dots,\hat{W}_3^m)^\top$ is some nonzero random vector lying in the eigenspace $E_{\lambda_2}$ of the replacement matrix $A$.

This concludes the proof.
\end{proof} 

\begin{rem}\label{rem 0<p_c^m< p_c^m'<p}
We refer to the case $0<p_c^m<p\le p_c^{m'}$ and $0<p_c^m< p_c^{m'}<p$. We first assume $0<p_c^m<p< p_c^{m'}$. By (\ref{superdiffusive relation a and p}), the term associated with $X$ which scales as $n^a$ dominates the term associated with $Y$ which scales as $\sqrt{n}$. Hence, as in the proof of Theorem~\ref{supersuperdiffusive V_0=2}, we obtain the limit random vector $L$ that depends only on the process $X$. The case $0<p_c^m<p=p_c^{m'}$ can be treated in the same way.

For the case $0<p_c^m< p_c^{m'}<p$, applying the scaling by $n^{a'}$ and the relation $p_c^m< p_c^{m'}\Rightarrow a<a'$, we observe that the term associated with $Y$ dominates the term associated with $X$, where $a'$ is given in (\ref{superdiffusive relation a and p}). Hence, we obtain the limit random vector $L$ that depends only on the process $Y$.
\end{rem}

\begin{rem}
From the proof of Theorem~\ref{supersuperdiffusive V_0=2}, we are interested in the distribution of $L$ given by (\ref{superdiffusive r.v. V_0=2}) and (\ref{superdiffusive r.v. V_0=1}). There is a substantial body of work on the asymptotic distribution of the superdiffusive ERW (cf. \cite{bercu2017martingale, bercu2019martierw, baur2016elephant, coletti2017central, guerin2023fixed, kubota2019gaussian, guerin2025onthelimit}). It is well known that, in contrast to the diffusive and critical regimes, the distribution of $L$ does
depend on the law of the initial step of the ERW. Gu\'{e}rin, Laulin and Raschel \cite{guerin2023fixed} provide detailed analyses of the distributional properties of the limiting random variable $L$, and also study in the higher-dimensional settings $\mathbb{Z}^d$. In this paper, we restrict our attention to the second moment, and it is expected that various further properties of $L$ could be derived by applying their framework, especially Theorem $3.8$ and Remark $3.5$ in \cite{bercu2019martierw}
and Lemma $3.2$ in \cite{guerin2023fixed}.

Assuming $\Bar{u}=\Bar{w}=\mathbf{0}$ and that the first step vectors $\sigma_1$ and $\tau_1$ are chosen uniformly from the $m$ directions, we obtain the following expressions for the first two moments
\begin{equation*}
    \mathbb{E}[L]=\frac{1}{\Gamma(a+1)}(\E[\sigma_1]+\E[\tau_1])=\mathbf{0},
\end{equation*}
\begin{equation*}
    \mathbb{E}[\Tilde{L}]=\frac{1}{\Gamma(a+1)}\E[\sigma_1]=\mathbf{0},
\end{equation*}
and
\begin{align*}
    \mathbb{E}[LL^\top]&= \frac{1}{\Gamma(2a+1)}(\E[\sigma_1\sigma_1^\top]+\E[\tau_1\tau_1^\top])+\frac{1}{(2a-1)\Gamma(2a+1)}\Sigma(\mathcal{U},\mathcal{W})\\
    &= \frac{1}{(2a-1)\Gamma(2a)}\Sigma(\mathcal{U},\mathcal{W}),
\end{align*}
\begin{align*}
    \mathbb{E}[\Tilde{L}\Tilde{L}^\top]&= \frac{1}{\Gamma(2a+1)}\E[\sigma_1\sigma_1^\top]+\frac{1}{(2a-1)\Gamma(2a+1)}\Sigma(\mathcal{U})\\
    &= \frac{1}{(2a-1)\Gamma(2a)}\Sigma(\mathcal{U}),
\end{align*}
where $\Sigma(\mathcal{U},\mathcal{W})$ is given by (\ref{Sigma_U,W}). We note that unlike in Type-I ERW, the distribution of $L$ for Type-II ERW depends on the behavior up to time $n=2$. For $\Tilde{L}$, when $\mathcal{U}=\mathcal{W}= \{\pm e_1, \pm e_2, \dots,\pm e_d\}$, we recover Theorem~$3.8$ in \cite{bercu2019martierw}.
\end{rem}

\section{Key quantities for typical examples}\label{section examples}

In this section, we present the explicit values of the key quantities derived in this paper for the representative examples, including those shown in Fig.~\ref{fig:lattices}.

\[
\begin{aligned}
&\text{(1) Standard lattice } \mathbb{Z}^d: 
&& \mathcal{U} = \mathcal{W} = \{\pm e_1, \pm e_2, \dots,\pm e_d\}, \ \\
&&& m=m'=2d,\ p_c^m=(2d+1)/(4d),\ \Bar{u}=\mathbf{0},\\
&&& a=(2dp-1)/(2d-1),\ C_a=(2d-1)/(1+2d-4dp),\\
&&& \Sigma(\mathcal{U})=\Sigma(\mathcal{W})=I_d/d.\\[4pt]
&\text{(2) Triangular lattice: } 
&& \mathcal{U} = \mathcal{W} = \{\pm u_1, \pm u_2, \pm u_3\},\\
&&&\text{}
u_1=(1,0)^\top,\ u_2=\left(\tfrac{1}{2},\tfrac{\sqrt{3}}{2}\right)^\top,\ 
u_3=\left(-\tfrac{1}{2},\tfrac{\sqrt{3}}{2}\right)^\top, \\
&&& m=m'=6,\ p_c^m=7/12,\ \Bar{u}=\mathbf{0},\\
&&& a=(6p-1)/5,\ C_a=5/(7-12p),\\
&&& \Sigma(\mathcal{U})=\Sigma(\mathcal{W})=I_2/2.\\[4pt]
&\text{(3) Hexagonal lattice: } 
&& \mathcal{U} = \{u_1,u_2,u_3\},\ 
\mathcal{W} = \{-u_1,-u_2,-u_3\},\\
&&&\text{} 
u_1=(0,1)^\top,\ 
u_2=\left(\tfrac{\sqrt{3}}{2},-\tfrac{1}{2}\right)^\top,\ 
u_3=\left(-\tfrac{\sqrt{3}}{2},-\tfrac{1}{2}\right)^\top, \\
&&& m=m'=3,\ p_c^m=p_c^{m'}=2/3,\ \Bar{u}=\Bar{w}=\mathbf{0},\\
&&& a=a'=(3p-1)/2,\ C_a=1/(2-3p),\\
&&& \Sigma(\mathcal{U})=\Sigma(\mathcal{W})=I_2/2.\\[4pt]
&\text{(4) Brick wall: } 
&& \mathcal{U} = \{\pm e_1, -e_2\},\quad 
\mathcal{W} = \{\pm e_1, e_2\},\\
&&& m=m'=3,\ p_c^m=2/3,\ \Bar{u}=\frac{1}{3}(0,-1)^\top,\ \Bar{w}=\frac{1}{3}(0,1)^\top,\\
&&& a=a'=(3p-1)/2,\ C_a=1/(2-3p),\\
&&& \Sigma(\mathcal{U})=\Sigma(\mathcal{W})=\frac{2}{9}\begin{pmatrix}
   3 & 0 \\
   0 & 1
\end{pmatrix}.
\end{aligned}
\]

We provide other examples that satisfy the conditions defined in Section~\ref{setting main}. 
\[
\begin{aligned}
&\text{(5)} 
&& \mathcal{U} = \{\pm u_1, \pm u_2\},\ \mathcal{W} = \{\pm e_1, \pm e_2\},\ u_1=(1,2)^\top,\ u_2=(1,-2)^\top,\\
&&& m=m'=4,\ p_c^{m}=5/8,\  \Bar{u}=\Bar{w}=\mathbf{0}, \ a=a'=(4p-1)/3,\\
&&&  C_{a}=3/(5-8p),\ \Sigma(\mathcal{U})=\begin{pmatrix}
   1 & 0 \\
   0 & 4
\end{pmatrix},\ \Sigma(\mathcal{W})=\frac{1}{2}I_2.
\end{aligned} 
\]

\[
\begin{aligned}
&\text{(6)} 
&& \mathcal{U} = \{\pm e_1, \pm e_2, (1,2)^\top\},\ \mathcal{W} = \{\pm e_1, \pm e_2\},\\
&&& m=5,\ m'=4,\ p_c^{m}=3/5,\ p_c^{m'}=5/8,\  \Bar{u}=\frac{1}{5}(1,2)^\top, \Bar{w}=\mathbf{0}, \\
&&&  a=(5p-1)/4,\ a'=(4p-1)/3,\ C_{a}=2/(3-5p),C_{a'}=3/(5-8p),\\
&&& \Sigma(\mathcal{U})=\frac{2}{25}\begin{pmatrix}
   7 & 4 \\
   4 & 13
\end{pmatrix},\ \Sigma(\mathcal{W})=\frac{1}{2}I_2.\\[4pt]
\end{aligned}
\]

\textbf{Acknowledgment.} I would like to express my deepest gratitude to my supervisor, Professor Tomoyuki Shirai, for his insightful discussions and guidance from the early stages of this project. This work was supported by WISE program (MEXT) at Kyushu University, and also supported the Kyushu University Fund ``Human Resource Development Initiative in Mathematics for Industry''.

\bibliographystyle{plain}
\bibliography{reference}
\end{document}